\documentclass[11pt]{article}
\usepackage{amsmath}
\usepackage{amsthm}
\usepackage{amssymb,amsfonts}
\usepackage[T1]{fontenc}
\usepackage{hyperref}
\usepackage{latexsym}
\usepackage{todonotes}
\usepackage{upgreek}
\usepackage{rotating}
\usepackage{nicefrac}
\usepackage{epsfig}
\usepackage{stmaryrd}
\usepackage{setspace}
\usepackage{enumerate}
\usepackage[all]{xypic}
\usepackage{bbm,ifpdf,tikz}
\ifpdf
\usepackage{pdfsync}
\fi

\newtheorem{theorem}{Theorem}[section]

\newtheorem{lemma}[theorem]{Lemma}
\newtheorem{proposition}[theorem]{Proposition}

{
\theoremstyle{definition}
\newtheorem{definition}[theorem]{Definition}
\newtheorem{example}[theorem]{Example}
\newtheorem{fact}[theorem]{Fact}

\newtheorem{remark}[theorem]{Remark}
\newtheorem{thmab}{Theorem}

}

\newcommand{\excise}[1]{}

\newcommand{\Vect}{\operatorname{Vec}}

\renewcommand{\dim}{\operatorname{dim}}

\renewcommand{\and}{\qquad\text{and}\qquad}

\newcommand{\Z}{\mathbb{Z}}
\newcommand{\Q}{\mathbb{Q}}
\newcommand{\N}{\mathbb{N}}
\newcommand{\R}{\mathbb{R}}

\newcommand{\cPGg}{\mathcal{PG}_{\!g}}

\newcommand{\cone}{{\operatorname{cone}}}

\DeclareMathOperator{\FI}{FI}

\DeclareMathOperator{\UConf}{UConf}

\begin{document}
\spacing{1.2}
\noindent{\LARGE\bf Hilbert series in the category of trees with contractions}\\

\noindent{\bf Eric Ramos}\\
Department of Mathematics, University of Oregon,
Eugene, OR 97403\\

{\small
\begin{quote}
We consider Hilbert series associated to modules over various categories of trees. Using the technology of Sam and Snowden \cite{sam}, we show that these Hilbert series must be algebraic. We then apply these technical theorems to prove facts about certain natural generating functions associated to trees.
\end{quote} }

\section{Introduction}\label{sec:intro}

\subsection{The setup}

Let $\mathcal{C}$ denote an essentially small category. Then a representation of $\mathcal{C}$ is a fuctor from $\mathcal{C}$ to the category of $\Q$ vector spaces. In their seminal work \cite{sam}, Sam and Snowden established the study of representations combinatorial categories; categories such as $\FI$, of finite sets and injections.  Their framework got at the combinatorial heart of the concurrent development of representation stability, due to Church, Farb, and Ellenberg \cite{CEF}\cite{CF}, while also expanding it in a variety of directions.

The language of Sam and Snowden, very broadly speaking, is useful for proving facts about a category's representations in two related, but distinct, realms. The first of these is related to the presence, or lack thereof, of a \emph{Noetherian property}. Just as with module over rings, one can make sense of \emph{finite generation} when discussing representations of categories (see Definition \ref{fgdef}) The Noetherian property asserts that submodules of finitely generated modules are also finitely generated. This is the theoretical backbone of virtually all of representation stability theory, as it allows one to prove finite generation statements about representations appearing in the limits of spectral sequences. The second tool granted by Sam and Snowden's work is a means by which one can understand \emph{Hilbert series} of finitely generated representations of one's category. 

To explain what is meant by this, let $M:\mathcal{C} \rightarrow \Vect_{\Q}$ be a finitely generated $\mathcal{C}$ representation, and assume that you have a function $\nu$ from the isomorphism classes of objects of $\mathcal{C}$ (which is guaranteed to be a set by our essential smallness assumption) to $\mathbb{N}$. For example, in the case of $\FI$, one may take $\nu$ to be the function which maps each set to its cardinality. Such a function is called a \emph{norm} of the category. Then the \emph{Hilbert series} of $M$ with respect to $\nu$ is the formal power series
\[
H_{M,\nu}(t) := \sum_{x} \dim_\Q(M(x)) t^{\nu(x)},
\]
where the sum is over isomorphism classes of objects.

Work of Miyata, Proudfoot, and the author applied Sam and Snowden's theory to a variety of categories of graphs \cite{PR-trees}\cite{PR-genus}\cite{MPR}. In all cases, a Noetherian property was proved for representations of the categories being considered. This was then applied to prove non-trivial consequences about homology groups of graph configuration spaces, as well as Khazdan-Lusztig coefficients of graphical matroids \cite{EPW}. Missing from this prior work, however, is a treatment of the Hilbert series of these representations. The goal of the present work is to bridge this gap in the literature, primarily for the category of planar rooted trees with contractions.

\subsection{The main theorem}

In this work, a \emph{tree} is a nonempty, at most 1-dimensional, connected, and finite CW-complex that is contractible. A \emph{contraction} is a continuous map between trees that involves contracting one or more edges of the tree while also possibly permuting the vertices (see Definition \ref{contraction}). A \emph{planar rooted tree} is a tree with a designated vertex (the \emph{root}) along with total orderings on the sets of edges coming out (e.g. away from the root) of every vertex. There is also a notion of planar contractions (see Definition \ref{contraction}), which are contractions that preserve all of various structures of the planar tree. The category of planar rooted trees and planar contractions is denoted $\mathcal{PT}$. We consider representations of the opposite category $\mathcal{PT}^{op}$. These representations were the focal point of the precursor work \cite{PR-trees}.

Before we can discuss the Hilbert series of these representations, we first must decide on a norm. In this paper, we will be working with the norm $\nu(T) = |E_T|$, where $E_T$ is the edge set of $T$. Therefore, for a finitely generated $\mathcal{PT}^{op}$-module $M$, one would like to consider
\[
H_{M}(t) = \sum_{T} \dim_\Q(M(T)) t^{|E_T|}.
\]
Before we do this, however, we first take the time to decide upon a "nice" enumeration of the isormophism classes of objects in $\mathcal{PT}^{op}$.

Recall the formal language of \emph{Dyck paths}. That is, the language whose alphabet is the set $\{u,d\}$, made up of words of even length, such that each of the characters $u$ and $d$ occupy exactly half of the word, and up to any $i$, the sub word of letters up to index $i$ has no more $d$'s than $u$'s. It is a fact that Dyck paths, which happen to be counted by the famous Catalan numbers, are in bijection with planar rooted trees. Given a Dyck path $w$, we read the word from left to right, letter by letter. Each time a $u$ is read, we take a step upward in the left-most (thusfar untraveled) direction, while every time a $d$ is read we step downward. Therefore, The word $ud$ corresponds to a single edge, while the word $uududd$ corresponds to the planar rooted tree that looks like the letter Y. For a Dyck path $w$, we will write $T_p(w)$ for the planar tree uniquely associated to $w$. We therefore write
\[
H_{M}(t) = \sum_{w} \dim_\Q(M(T_p(w))) t^{l(w)/2},
\]
where $l(w)$ is the length of the path $w$.

\begin{thmab}\label{mainthm}
Let $M$ be a finitely generated $\mathcal{PT}^{op}$-module. Then the Hilbert series
\[
H_M(t) = \sum_{w} \dim_\Q(M(T(w))) t^{l(w)/2},
\]
is algebraic.
\end{thmab}

The proof philosophy we apply for Theorem \ref{mainthm} follows the lingual category approach of Sam and Snowden \cite{sam}. In particular, we show that the category of planar rooted trees and contractions is unambiguous and context-free. In the final section of this work, we apply Theorem \ref{mainthm} to prove certain natural generating functions associated to Dyck paths are algebraic. These applications are novel (to the best knowledge of the author), and should be of some independent interest.

Because planar trees are just trees with extra structure, to every Dyck path $w$ we can associate a tree (not planar or rooted), which we call $T(w)$. Note that this association does not uniquely recover the Dyck path, but every tree arises in this way. Writing $\mathcal{T}$ for the category of trees and contractions, and given a finitely generated $\mathcal{T}^{op}$-module $M$, we define its \emph{Hilbert-Dyck} series as the formal power series
\[
HD_M(t) := \sum_{w} \dim_\Q(M(T(w))) t^{l(w)/2},
\]
where the sum is over all Dyck paths, and $l(w)$ is the length of the path. By how the association $w \mapsto T(w)$ is defined, we observe that $l(w)/2 = |E_T|$. Moreover, because this association is not a bijection, the Hilbert-Dyck series is \textbf{not} equal to the usual Hilbert series of modules over this category. Indeed, one may write
\[
HD_M(t) = \sum_{T} p_{T} \dim_\Q(M(T))t^{|E_T|},
\]
where $p_T$ is the total number of Dyck paths which correspond to the tree $T$. Unlike the aforementioned Hilbert series of $M$, the Hilbert-Dyck series will prove to be much more tractable. For instance, assuming $M$ is the module which assigns $\Q$ to every tree, one has
\[
HD_M(t) = \sum_{T} p_{T} t^{|E_T|} = \sum_{n} c_n t^n
\]
where $c_n$ is the $n$-th Catalan Number. This generating function is far better understood than the generating function for the number of isomorphism classes of trees. For instance, it is a celebrated fact that this generating function is algebraic. Our second technical result is that this is the case in general.

\begin{thmab}\label{mainthm2}
Let $M$ be a finitely generated $\mathcal{T}^{op}$-module. Then the Hilbert-Dyck series
\[
HD_M(t) = \sum_{w} \dim_\Q(M(T(w))) t^{l(w)/2},
\]
is algebraic.
\end{thmab}

One nice property of algebraic generating functions is their asymptotics are fairly predictable. For instance, one has the following fact.

\begin{fact}
Let $(f_n)_{n \geq 0}$ be a sequence of natural numbers such that $F(t) := \sum_{n \geq 0} f_n t^n$ is an algebraic function. Further assume that $F(t)$ has a unique singularity at its radius of convergence. Then there exist constants $C,\rho \in \R, \alpha \in \Q$, such that $f_n$ is asymptotically close to $Cn^\alpha \rho^n$. That is to say, 
\[
\lim_{n \to \infty} \frac{f_n}{Cn^\alpha \rho^n} = 1.
\]
\end{fact}

\begin{remark}
The requirement that the generating function has a unique singularity on its radius of convergence is not strictly necessary, though the statement is more complicated if we do not assert it. In this more technical case, the ultimate conclusion is only true up to residue classes of $n$ (see \cite{BD}).
\end{remark}

This fact can be observed explicitly in the case where $f_n = c_n$ is $n$-th Catalan number. In this case we have, from Sterling's approximation,
\[
f_n \cong \frac{1}{\sqrt{\pi}} n^{-3/2} 4^n.
\]

Coming back to the context of Theorem \ref{mainthm}, we see that if you look at the total dimension
\[
\sum_{|E_T| = n} \dim_\Q M(T)
\]
then (possibly up to the residue class of $n$), it grows at worst like some power of $n$ times an exponential. That is, it grows at worst exponentially. This is consistent with the fact, proven in \cite{PR-trees}, that each individual vector space $M(T)$ is bounded by a polynomial in the number of edges of $T$, whenever $T$ is sufficiently large. Unfortunately the techniques of the current paper do not immediately recover the constants $C,\alpha,$ and $\rho$. It would be interesting to see whether this can be done in general, and how they compare to the analogous constants for the generating function of the Catalan numbers.

\subsection{Other categories of graphs}

The current work mainly considers the category of trees with contractions. However, Because edge contractions are homotopy equivalences, they preserve the first Betti-number, or \emph{genus}, of the graph. This shows that the category of all graphs and contractions is stratified by this genus invariant, where the tree case is only one stratum. The work \cite{PR-genus} shows that other strata of the category of all graphs and edge contractions are also of great interest.

The techniques of this paper will generalize to these other strata, although statements of theorems become considerably more difficult. Indeed, by looking at spanning trees, one may think of a higher genus graph as being a tree decorated with the data of how the extra edges are attached to each vertex. This is exemplified in the category $\cPGg$ discussed in \cite{PR-genus}. Instead of working with Dyck paths, one instead must work with Dyck paths that are decorated with this finite amount of extra data. In particular, one may define generalized Hilbert-Dyck series and prove they are algebraic. To the author's knowledge, these decorated Dyck paths have not appeared in the literature, and it is therefore unclear whether they are of any particular interest. For this reason, we do not pursue this direction further.

That being said, however, it is certainly possible that these more general categories of graphs have differently defined Hilbert series that admit nice formulas. We leave this as an avenue for possible future research. There is particular interest in understanding Hilbert series of the Graph minor category, as described in \cite{MPR}.

\section*{Acknowledgements}
The author was supported by NSF grant DMS-1704811. He would like to send thanks to Ben Young for various conversations that were useful during the creation of this work. He would also like to send thanks to Nick Proudfoot, whose editorial suggestions vastly improved the quality of the writing.

\section{Background}

\subsection{Categories of trees}

In this section, we outline the three main categories whose representations will be studied in this work. Most of what follows can be found in \cite{PR-trees}, and \cite{Barter}.

\begin{definition}\label{contraction}
A \textbf{tree} is a one-dimensional contractible CW-complex. A \textbf{rooted tree} is a tree paired with a choice of vertex called the \textbf{root}. This choice of root implicitly directs the edges of the tree away from the root. A \textbf{planar rooted tree}, or just a \textbf{planar tree}, is a rooted tree equipped with well-orderings on the sets of edges leaving each vertex. Planar trees have a natural well-ordering on their vertices via a depth-first search from the root.

Given trees $T,T'$, a \textbf{contraction} from $T$ to $T'$ is a map of sets
\[
\varphi:V_T \sqcup E_T \rightarrow V_{T'} \sqcup E_{T'}
\]
satisfying:
\begin{itemize}
\item $\varphi(V_T) = V_{T'}$;
\item for every $e' \in E_{T'}$ there exists a unique edge $e \in E_T$ with $\varphi(e) = e'$;
\item for every $e = \{x,y\} \in E_T$, if $\varphi(e) = v' \in V_{T'}$ then $\varphi(x) = \varphi(y) = v'$, while if $\varphi(e) = e' \in E_{T'}$ then $e' = \{\varphi(x),\varphi(y)\}$;
\item for every $v' \in V_{T'}$, the preimage $\varphi^{-1}(v') \subseteq V_T \sqcup E_T$ consists of the edges and vertices of some subtree of $T$.
\end{itemize}
Given two rooted trees, a \textbf{rooted contraction} between them is a contraction of the underlying trees which preserves the root. Finally, a \textbf{planar contraction} between planar trees $\varphi:T \rightarrow T'$ is a rooted contraction with the property that given two vertices $v_1',v_2' \in V_{T}$ such that $v_1' < v_2'$ in the depth-first order, one has that the vertex in $\varphi^{-1}(v_1')$ closest to the root is smaller than the vertex in $\varphi(v_2')$ closest to the root, in the depth-first order.

Finally, we will write $\mathcal{T}$ for the category of trees with contractions, $\mathcal{RT}$ for the category of rooted trees with rooted contractions, and $\mathcal{PT}$ for the category of planar trees with planar contractions.
\end{definition}

In this paper, we will be largely concerned with the representation theory of the categories $\mathcal{T}$, $\mathcal{RT}$, and $\mathcal{PT}$. The study of such objects was essentially initiated by Barter \cite{Barter}, although a different language was used in that work. In the precursers to the current paper \cite{PR-trees, PR-genus}, Proudfoot and the author prove that the categories presented above are equivalent to those considered by Barter.

\begin{definition}\label{fgdef}
Let $\mathcal{C}$ denote anyone of the categories $\mathcal{T}$, $\mathcal{RT}$, or $\mathcal{PT}$. Then a \textbf{representation of $\mathcal{C}^{op}$} or a \textbf{$\mathcal{C}^{op}$-module} is a contravariant functor
\[
M: \mathcal{C} \rightarrow \Vect_{\Q}
\]
where $\Vect_{\Q}$ is the category of finite dimensional vector spaces over $\Q$. Equivalently, a $\mathcal{C}^{op}$-module is a functor
\[
M: \mathcal{C}^{op} \rightarrow \Vect_{\Q}.
\]

We say that a $\mathcal{C}^{op}$-module $M$ is \textbf{finitely generated} if there exists a finite list of trees (or rooted trees, or planar trees) $\{T_i\}_{i \in I}$ such that for any tree $T \notin \{T_i\}_{i \in I}$, the vector space $M(T)$ is spanned by the images
\[
M(\varphi):M(T_i) \rightarrow M(T),
\]
where $\varphi:T \rightarrow T_i$ is a contraction. We call the trees $\{T_i\}$ the \textbf{generators} of the module $M$, and say $\{T_i\}$ \textbf{generates} $M$.
\end{definition}

\begin{remark}
The category of $\mathcal{C}^{op}$-modules is abelian, with the standard abelian operations defined point-wise. In particular, we can reuse terms from the language of modules over a ring without ambiguity.
\end{remark}

One of the most important properties of finitely generated $\mathcal{C}^{op}$-modules is the Noetherian property.

\begin{theorem}[\cite{Barter}, \cite{PR-trees}]
If $M$ is a finitely generated $\mathcal{C}^{op}$-module, then all submodules of $M$ are also finitely generated.
\end{theorem}

In this paper we consider the types of growth that can appear in the dimensions of the vector spaces $M(T)$. This question was partially considered in the precursor work \cite{PR-trees}, where the following is proved.

\begin{theorem}[\cite{PR-trees}]
Let $M$ be a finitely generated $\mathcal{C}^{op}$-module. Then there exists a polynomial $P_M(t) \in \Q[t]$ such that for all trees with $|E_T| \gg 0$, one has
\[
\dim_\Q(M(T)) \leq P_M(|E_T|).
\]
\end{theorem}

\cite{PR-trees} also proves results which show how this polynomial behavior is sharp, so long as you vary the trees within certain natural families of trees. In this work we consider growth as it pertains to the module $M$ as a whole, instead of how it pertains to the individual vector spaces which comprise it.

\begin{definition}\label{hildef}
It is a well known fact that planar rooted trees with $n$ edges are in bijection with \textbf{Dyck paths of length $2n$}. A Dyck path is a word of even length $2n$ in the alphabet $\{u,d\}$ such that each of the characters $u$ and $d$ appear exactly $n$ times and up to any $i$, the sub word of letters up to index $i$ has no more $d$'s than $u$'s.

Given a Dyck path $w$ of length $2n$, we write $T(w)$ (resp. $T_r(w)$, resp. $T_p(w)$) to denote the tree (resp. rooted tree, resp. planar rooted tree) associated to $w$. Note that $T(w)$ and $T_r(w)$ do not uniquely determine the original word $w$, though every tree and rooted tree can be written in this form for some $w$.

Let $M$ denote a finitely generated $\mathcal{T}^{op}$-module. Then the Hilbert-Dyck series associated to $M$ is the formal power series
\[
HD_M(t) = \sum_{w} \dim_\Q(M(T(w))) t^{|E(T(w))|},
\]
where the sum is over all Dyck paths. Note that $|E(T(w))| = l(w)/2$, where $l(w)$ is the length of the word $w$. We similarly define Hilbert-Dyck series for modules over the category $\mathcal{RT}$.
\end{definition}

\begin{example}
Consider the $\mathcal{T}^{op}$-module which assigns to every tree the vector space $\Q$, and to every contraction the identity map. This is sometimes referred to as the trivial $\mathcal{T}^{op}$-module. Then we have
\[
HD_M(t) = \sum_{n \geq 1} c_{n} t^n,
\]
where $c_{n}$ is the number of Dyck paths of length $n$, i.e. the $n$-th Catalan number. In particular, $HD_M(t)$ is precisely the generating function for the Catalan numbers.

Note that, if instead $M$ was the $\mathcal{RT}$-module (resp. $\mathcal{PT}$-module) which assigns $\Q$ to every rooted (resp. planar rooted) tree, then $HD_M(t)$ (resp. $H_M(t)$) is identical to the above.
\end{example}

It is a well-known fact that the generating function for the Catalan numbers is \textbf{algebraic}. That is, it satisfies a polynomial equation with coefficients in $\Q(n)$. Our main result can therefore be seen as a categorification of this fact. See \cite{BM} for a comprehensive treatment of algebraic generating functions and their applications.

\subsection{PDA's and context-free languages}

In this section we discuss the theory of Push-down Automata (PDA) and their associated context-free languages. See \cite{ABB} for a standard reference. Before we dive into the somewhat intimidating formalities of the subject, we take a moment to try to develop the basic intuition for what PDAs are designed to accomplish.

\begin{definition}
Let $\Sigma$ be a finite set. Then we define the \textbf{Kleene star} $\Sigma^\ast$ to be the free monoid generated by the set $\Sigma$. A \textbf{language $\mathcal{L}$ with alphabet $\mathbf{\Sigma}$} is just any subset of $\Sigma^\ast$. Given a word $w \in \mathcal{L}$, we write $l(w)$ to denote the \textbf{length of $w$}. That is, the number of elements of $\Sigma$ which appear in $w$.

In this paper, we follow the standard practice of the field and reserve the symbol $\epsilon$ to denote the empty word.
\end{definition}

\begin{remark}
Much of what follows will actually work for any \textbf{norm} on the language $\mathcal{L}$, not just the length. This level of generality will not be necessary for us.
\end{remark}

Complexity in language theory is concerned with two distinct, but essentially equivalent, perspectives. The first perspective is the question of how complicated a grammar needs to be in order to build the language from its alphabet. The second perspective is the question of how sophisticated a machine needs to be to be able to detect whether a given word is in the language. The simplest possible machines are finite state automata. These machines have finitely many states, and a finite list of rules which allow one to move between states given an input element of $\Sigma$. The kinds of languages whose inclusion problem can be solved by finite state automata are the so-called \textbf{regular languages}. In this paper we will largely be concerned with machines that are one step higher in complexity: finite automata equipped with memory in the form of a stack.

\begin{definition}
A \textbf{push-down automaton}, or \textbf{PDA}, is a 7-tuple $P = (Q,\Sigma,\Gamma,\delta,q_0,Z,F)$, where:
\begin{itemize}
\item $Q$ is a finite set called the \textbf{states} of $P$;
\item $\Sigma$ is a finite set, disjoint from $Q$, called the \textbf{alphabet} of $P$;
\item $\Gamma$ is a finite set, disjoint from $\Sigma$ and $Q$, called the \textbf{stack symbols} of $P$;
\item $\delta: Q \times (\Sigma \cup \{\epsilon\}) \times (\Gamma \cup \{\epsilon\}) \rightarrow \mathcal{P}(Q \times \Gamma^\ast)$, where $\mathcal{P}$ denotes the power set, is the \textbf{transition function} of $P$;
\item $q_0 \in Q$ is the \textbf{initial state} of $P$;
\item $Z \in \Gamma$ is the \textbf{initial stack symbol} of $P$;
\item $F \subseteq Q$ is the set of \textbf{final states} of $P$.
\end{itemize}

An \textbf{instantaneous description} of $P$ is a triple, $(q,w,S) \in Q \times \Sigma^\ast \times \Gamma^\ast$. We interpret an instantaneous description as telling us which state we are currently in, the remainder of the word that is currently being processed, and the contents of the stack, where we understand the left most symbol of $S$ as being the top of the stack. If $(q,aw,AS)$ is an instantaneous description with $a \in \Sigma \cup \{\epsilon\}$ and $A \in \Gamma \cup \{\epsilon\}$, then we write
\begin{eqnarray}
(q,aw,AS) \mapsto (q',w,\alpha S) \label{simplemove}
\end{eqnarray}
if $(q',\alpha) \in \delta(q,a,A)$. More generally, if $(q,w,S)$ and $(q',w',S')$ are two instantaneous descriptions of $P$, then we write
\[
(q,w,S) \mapsto (q',w',S')
\]
if there is a series of moves of the form (\ref{simplemove}) transforming $(q,w,S)$ into $(q',w',S')$. Finally, we say that $P$ \textbf{recognizes} a word $w \in \Sigma^\ast$ if
\[
(q_0,w,Z) \mapsto (q_r,\epsilon,S)
\]
where $q_0$ and $Z$ are the initial state and stack symbol, respectively, $S \in \Gamma^\ast$, and $q_r \in F$ is a final state. The \textbf{language $\mathcal{L}(P)$ of $P$} is the set of all words that are recognized by $P$. We call $P$, as well as the language $\mathcal{L}(P)$, \textbf{unambiguous} if for any $w \in \mathcal{L}(P)$ there is precisely one sequence of moves of the form (\ref{simplemove}) that leads to a final state.
\end{definition}

We think of a push-down automaton $P = (Q,\Sigma,\Gamma,\delta,q_0,Z,F)$ as being a machine that inputs a word $w$ and outputs either "yes" or "no." It does so in the following way: Writing our word as $w = aw'$, with $a \in \Sigma$, it reads the letter $a$ as well as the top of the stack, $Z$, and checks its available moves, as prescribed by $\delta(q_0,a,Z)$. These moves may include popping the top of the stack, pushing more symbols onto the stack, or some combination of both, along with a possible jump to a new state. If there are no available moves, then the machine outputs "no." Otherwise, it continues reading the remaining word $w'$ in this way. When the entire word has been read, if the machine is in a final state it outputs "yes," while otherwise it outputs "no."

One should observe that the transition function of a PDA is permitted to read $\epsilon$ for either the input letter or the top stack symbol. This does \emph{not} signify that the input word or stack must be empty for this transition to occur. It is more correct to interpret these transitions (sometimes called $\epsilon$-moves in the literature) as saying that this transition can happen regardless of what the next input letter (or top of the stack) is. We will see examples of these kinds of transitions during the proof of the main theorem.\\

Languages which are of the form $\mathcal{L}(P)$ for some PDA $P$ are called \textbf{context-free}. Importantly for us, one has the following foundational result about context-free languages.

\begin{definition}\label{LangHS}
Let $\mathcal{L}$ be a language over some finite alphabet $\Sigma$. Then the \textbf{generating function of $\mathcal{L}$} is the formal power series
\[
H_{\mathcal{L}}(t) := \sum_{w \in \mathcal{L}} t^{l(w)},
\]
where $l(w)$ is the length of the word $w \in \mathcal{L}$.
\end{definition}

\begin{theorem}[Proposition 3.7, \cite{BM}]
Let $\mathcal{L}$ be a context-free language associated to an unambiguous PDA. Then $H_{\mathcal{L}}(t)$ is algebraic.
\end{theorem}

\begin{remark}
We will see that Hilbert-Dyck series are in fact always $\Z$-algebraic (See \cite{BD}). We do not make use of this distinction in this paper. 
\end{remark}

\begin{example}\label{cataex}
We have already seen that the Catalan numbers have an algebraic generating function. In fact, we can realize the Catalan numbers as the number of words of a given length in an unambiguous context-free language as follows.

Set $P =(Q,\Sigma,\Gamma,\delta,q_0,Z,F)$, where $Q = \{q_0,q_1\}$, $\Sigma = \{u,d\}$, $\Gamma = \{Z,A\}$, and $F = \{q_1\}$. Our transition function will be defined by the assignments:
\begin{align*}
\delta(q_0,u,Z) = (q_0,AZ)\\
\delta(q_0,u,A) = (q_0,AA)\\
\delta(q_0,d,A) = (q_0,\epsilon)\\
\delta(q_0,\epsilon,Z) = (q_1,\epsilon)
\end{align*}
Note that we follow the standard practice in the field that when the output of the transition function is a singleton, we suppress the set notation. Moreover, any transition whose output is the empty set is not written.

In words, the first two transitions indicate that when a $u$ is read by the PDA, the symbol $A$ is added to the top of the stack, while the third indicates that if a $d$ is read, the stack is popped. Finally, the last transition indicates that, at any time when the stack only contains the initial symbol, you may move on to the final state. It is clear from this description that $P$ is unambiguous, as the transition function has at most one move for any legal input. Moreover, a quick argument shows that $\mathcal{L}$ is precisely the language of Dyck paths. Our claim then follows from the fact that the number of Dyck paths of a given length agrees with the Catalan numbers.
\end{example}

\subsection{Lingual categories}
In their seminal work \cite{sam}, Sam and Snowden develop a kind of language theory for categories, which they call lingual categories. Roughly speaking, these are categories whose morphisms can be encoded as "well-behaved" languages. The upshot to this is one can use well known combinatorial theorems about the Hilbert series of these languages (See Definition \ref{LangHS}) to conclude non-trivial facts about the dimension growth of modules over the category. In particular, we have the following.

\begin{definition}\label{contextfreecat}
Let $\mathcal{C}$ denote an essentially small category with no non-trivial endmorphisms, and write $x$ for an object of $\mathcal{C}$. Then we write $|\mathcal{C}_x|$ for the set
\[
|\mathcal{C}_x| = \{f:x \rightarrow y \mid y \text{ is an object of $\mathcal{C}$}\} / \sim ,
\]
where $\sim$ is the relation
\[
f \sim g \iff h \circ f = g \text{ for some isomorphism $h$.}
\]
The set $|\mathcal{C}_x|$ can be enhanced with the structure of a poset, with order relation given by
\[
f \leq g \iff g = h \circ f \text{ for some morphism $h$.}
\]

We say that the category $\mathcal{C}$ is an \textbf{unambiguous and context-free} if the following four conditions hold:
\begin{itemize}
\item the category $\mathcal{C}$ is \textbf{Gr\"obner} in the sense of Sam and Snowden \cite{sam};
\item for every object $x$, there exists a set theoretic bijection
\[
\iota_x:|\mathcal{C}_x| \cong \mathcal{L}_x,
\]
where $\mathcal{L}_x$ is an unambiguous context-free language;

\item for every object $x$, and every order ideal $I$ of the poset $|\mathcal{C}_x|$, the image $\iota_x(I) \subseteq \mathcal{L}_x$ is also an unambiguous context-free language;

\item there exists a function $\nu$, called the \textbf{norm} of $\mathcal{C}$, from the set of isomorphism classes of objects of $\mathcal{C}$ to $\N$ such that for any object $x$ and any morphism $f:x \rightarrow y$,
\[
\nu(y) = l(\iota_x(f)).
\]
\end{itemize}
\end{definition}

The theory of Gr\"obner categories was developed by Sam and Snowden. One can think of this condition as saying that the representation theory of the category admits a theory of Gr\"obner bases. This notation was extended by Miyata, Proudfoot, and the author in \cite{MPR} to modules over categorical algebras. For the purposes of the present work, just note that the category $\mathcal{PT}^{op}$ was proved to be Gr\"obner by Barter in \cite{Barter}. We therefore do not need to worry too much about this condition going forward.

\begin{theorem}[\cite{sam}]\label{hilser}
Let $\mathcal{C}$ be an unambiguous context-free category with norm $\nu$, and let $M$ be a $\mathcal{C}$-module. If $M$ is finitely generated, then the formal power series
\[
H_{M,\nu}(t) := \sum_{x} \dim_\Q(M(x))t^{\nu(x)}
\]
is algebraic.
\end{theorem}

In view of Theorem \ref{hilser}, and Definition \ref{hildef}, our path forward has now become clear. Our first step will be to prove that the category $\mathcal{PT}^{op}$ is unambiguous and context-free, thereby generalizing the computation in Example \ref{cataex}. This will imply that the Hilbert series for finitely generated modules over $\mathcal{PT}$ are algebraic by Theorem \ref{hilser}. Following this, we leverage the fact that the forgetful functors $\mathcal{PT} \rightarrow \mathcal{RT}$ and $\mathcal{PT} \rightarrow \mathcal{T}$ have the so-called property (F) (see \cite{PR-trees} and \cite{sam}). In particular, pulling back any finitely generated $\mathcal{RT}^{op}$ or $\mathcal{T}^{op}$-module to a module over $\mathcal{PT}^{op}$ preserves finite generation. This will imply that Hilbert-Dyck series of finitely generated modules over the categories $\mathcal{RT}^{op}$ and $\mathcal{T}^{op}$ will be algebraic, as desired.

\section{The proof of the main theorem}

In this section, we prove our main theorem via the strategy just outlined. In particular, our ultimate goal is the prove the following.

\begin{theorem}\label{mainlangthm}
The category $\mathcal{PT}^{op}$ is unambiguous and context-free, with norm given by
\[
\nu(T) = 2\cdot (\# \text{ of edges of } T).
\]
\end{theorem}

\begin{remark}
We note that, with the norm defined as it is above, the associated Hilbert series are not exactly the previously defined Hilbert series (Definition \ref{hildef}). However, they are related by substituting $t$ for $\sqrt{t}$. This operation clearly preserves the ultimate conclusion that the Hilbert series are algebraic, and we therefore stick with the aforementioned norm so that Lemma \ref{normisok} remains true.
\end{remark}

Proving this theorem happens in three steps. To begin, we must first decide on a means of encoding the morphisms of $\mathcal{PT}$ as words in a language.\\

\emph{For the remainder of this section, we fix a planar rooted tree $T$ with $n$ vertices.}\\

\begin{definition}
We may assume that the vertices of $T$ have been identified with $\{0,\ldots,n-1\}$. Then we define the alphabet $\Sigma_T$ to be the finite set of symbols
\[
\Sigma_T := \{u_i,d_i \mid i \in \{0,\ldots,n-1\}\}.
\]
Thus, $|\Sigma_T| = 2n$. We will encode $|\mathcal{PT}_T|$ as a language over the alphabet $\Sigma_T$.

Let $T'$ be a planar rooted tree, and let $\phi:T \rightarrow T'$ be an opposite planar contraction, with associated contraction $\psi:T' \rightarrow T$. Then we encode $\phi$ as a word in $\Sigma_T$ as follows. Let $w$ be the Dyck path associated to the tree $T'$. We add subscripts to the $u$'s and $d$'s in this path by looking at $\psi$ applied to the head (directed, as always, away from the root) of the associated directed edge when a $u$ is read, and the tail of the associated edge when a $d$ is read. We will write $w_\phi$ to denote this word.

Finally, we write $\mathcal{L}_T$ for the language
\[
\mathcal{L}_T := \{w_\phi \mid \phi:T \rightarrow T'\}.
\]

\end{definition}

\begin{example}
To see an example of the above encoding, let $T'$ be the planar tree pictured in Figure \ref{treefig}. Then the Dyck path associated to $T'$ is given by
\[
uuuudduuuudduudddddd.
\]
Assume now that $T$ is a single edge, with root and head labeled by 0 and 1, respectively, and let $\psi:T' \rightarrow T$ be the planar contraction which sends the vertices labeled 9 and 10 to the head of $T$, and all other vertices to the root of $T$. Then,
\[
w_{\phi} = u_0u_0u_0u_0d_0d_0u_0u_0u_0u_0d_0d_0u_1u_1d_1d_1d_0d_0d_0d_0.
\]
\end{example}

\begin{remark} \label{biggerwords}
The idea to encode these morphisms as modified Dyck paths was first accomplished by Barter in \cite{Barter}, where they were called Catalan words. Our encoding is different from his, but the basic premise is the same.\\

Also note that if $\zeta:V(T') \rightarrow V(T)$ is any function of sets (not necessarily a planar contraction), then one can similarly make sense of a word on the alphabet $\{u_i,d_i\}$ corresponding to $\zeta$. We will use this observation during the proof of the main theorem.
\end{remark}

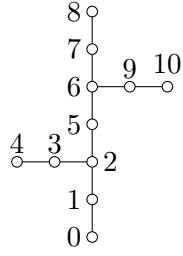
\begin{figure}
\centering
\begin{tikzpicture}
    \tikzstyle{every node}=[draw,circle,fill=white,minimum size=4pt,
                            inner sep=0pt]
    \draw (0,0) node (0) [label=above:$4$] {}
		      --(.5,0) node (1) [label=above:$3$]{}
					--(1,0) node (2) [label=right:$2$]{}
					--(1,.5) node (3) [label=left:$5$]{}
					--(1,1) node (4) [label=left:$6$]{}
					--(1,1.5) node (5) [label=left:$7$]{}
					--(1,2) node (6) [label=left:$8$]{}
    			(4) -- (1.5,1) node (7) [label=above:$9$]{}
					--(2,1) node (8) [label=above:$10$]{}
					(2) -- (1,-.5) node (9) [label=left:$1$]{}
					--(1,-1) node (10) [label=left: $0$]{};

\end{tikzpicture}

\caption{A planar tree $T'$}\label{treefig}
\end{figure}

It was already proven in \cite{Barter} that $\mathcal{PT}$ is Gr\"obner. This resolves the first condition in Definition \ref{contextfreecat}. We will now prove that $\mathcal{L}_T$ is always an unambiguous context-free language, thus verifying the second condition of Definition \ref{contextfreecat}.

\begin{proposition}\label{almostUCF}
The language $\mathcal{L}_T$ is an unambiguous context-free language.
\end{proposition}

\begin{proof}
Our goal will be to produce an unambiguous PDA, $P_T$, whose associated language is $\mathcal{L}_T$. We define the components of this PDA in turn as follows:

\begin{itemize}
\item The states of the PDA $Q_T$ are given by the initial state $q_0$, the final state $q_f$, as well as a pair of states $q_{(e,u)}$ and $q_{(e,d)}$ for every edge $e$ of $T$.
\item The alphabet of the PDA is $\Sigma_{T}$, while the stack alphabet $\Gamma_T$ contains the initial symbol $Z$, as well as symbols $A_{v}$ for every vertex $v$ of $T$.
\end{itemize}

To finish the construction of $P_T$, we need to detail our transition relations. We accomplish this by examining a handful of cases, which condition on the state we are currently situated at.\\

\textbf{CASE: Transitions originating from the initial state $q_0$.}\\

In this case we have

\begin{eqnarray*}
\delta(q_0,u_0,\epsilon) &=& (q_0,A_0A_0)\\
\delta(q_0,d_0,A_0) &=& (q_0,\epsilon)\\
\delta(q_0,u_1,\epsilon) &=& (q_{(e_1,u)},A_1), \text{ where $e_1$ is the first edge leaving the root.}\\
\end{eqnarray*}

In other words, in this state the PDA can read either $u_0$, $d_0$, or $u_1$. while it is reading the symbols $u_0$ and $d_0$, it essentially acts as the PDA which recognizes the language of Dyck paths (see Example \ref{cataex}). If it reads the symbol $u_1$, however, it moves to the first non-initial state, while adding an $A_1$ to the top of the stack.\\

\textbf{CASE: Transitions originating from the state $q_{(e,u)}$, where $e$ is some edge of $T$ whose head (directed away from the root) is the vertex $v$.}\\

This case has two subcases. Firstly, assume that the vertex $v$ is not a leaf, and that the smallest edge leaving $v$ is $e'$, with head $v'$. In this subcase we see,

\begin{eqnarray*}
\delta(q_{(e,u)},u_v,A_v) &=& (q_{(e,u)},A_vA_v)\\
\delta(q_{(e,u)},d_v,A_v) &=& (q_{(e,u)},\epsilon)\\
\delta(q_{(e,u)},u_{v'},A_v) &=& (q_{(e',u)},A_{v'}A_v).
\end{eqnarray*}

Note that these transitions are essentially the same as in the previous case, with one somewhat subtle difference. While this state can accept the letters $u_v$ and $u_{v'}$, it can only do so if the top of the stack displays the symbol $A_v$. The reason for this is that, based on the first two transitions, it is technically possible for a sufficient number of $d_v$ symbols to be read so as to completely pop $A_v$ off the stack. If one were to then try to add either an $A_v$ or an $A_{v'}$ to the stack (e.g. by reading a $u_v$ or $u_{v'}$-respectively), the input word could not possible be coming from a contraction. Indeed, if a $u_v$ is read at this point, then the vertex map associated to the input word (see Remark \ref{biggerwords}) would have a disconnected preimage at $v$. If a $u_{v'}$ is read, then the vertex map does not preserve edge adjacency, and is therefore not a contraction either. These transitions are therefore modified to save us from accepting such a word.

In the second subcase, we assume that $e$ is a leaf, and that the tail of this leaf is $v'$. We have

\begin{eqnarray*}
\delta(q_{(e,u)},u_v,A_v) &=& (q_{(e,u)},A_vA_v)\\
\delta(q_{(e,u)},d_v,A_v) &=& (q_{(e,u)},\epsilon)\\
\delta(q_{(e,u)},\epsilon ,A_{v'}) &=& (q_{(e,d)},A_{v'}).
\end{eqnarray*}

This subcase is similar to the previous. Because we have assumed that $e$ is a leaf, there is nowhere to go but back down to $v'$. If at any point the top of the stack displays the symbol $A_{v'}$, then we have closed off all of the $u_v$ letters in our word, and must now move back down the tree $T$ to proceed with our mapping.\\

\textbf{CASE: Transitions originating from the state $q_{(e,d)}$, where $e$ is some edge of $T$ whose tail (directed away from the root) is the vertex $v$.}\\

Once again we have a few subcases. In the first subcase, we assume that $v$ is not the root, and that there is some edge $e'$, with head $v'$, which is the smallest edge outgoing from $v$ for which $A_{v'}$ has never appeared on the stack. We have

\begin{eqnarray*}
\delta(q_{(e,d)},u_v,A_v) &=& (q_{(e,d)},A_vA_v)\\
\delta(q_{(e,d)},d_v,A_v) &=& (q_{(e,d)},\epsilon)\\
\delta(q_{(e,d)},u_{v'} ,A_{v}) &=& (q_{(e',u)},A_{v'}A_v).
\end{eqnarray*}

In our second subcase, we assume that $v$ is still not the root, no such edge $e'$ exists, that $e''$ is the incoming edge of $v$, and that $v''$ is the other endpoint of $e''$. Further assume that $v''$ is also not the root. In the context of planar contractions, we will be in this case when we have already resolved how we are going to map the vertices of $T'$ to the vertices of $T$ above $v$. All that remains is to finish mapping vertices to $v$, and move back down the tree $T$. We have

\begin{eqnarray*}
\delta(q_{(e,d)},u_v,A_v) &=& (q_{(e,d)},A_vA_v)\\
\delta(q_{(e,d)},d_v,A_v) &=& (q_{(e,d)},\epsilon)\\
\delta(q_{(e,d)},\epsilon ,A_{v''}) &=& (q_{(e'',d)},A_{v''}).
\end{eqnarray*}

Repeating the previous subcase, but assuming that $v''$ is the root we have,

\begin{eqnarray*}
\delta(q_{(e,d)},u_v,A_v) &=& (q_{(e,d)},A_vA_v)\\
\delta(q_{(e,d)},d_v,A_v) &=& (q_{(e,d)},\epsilon)\\
\delta(q_{(e,d)},\epsilon ,A_{0}) &=& (q_{(e'',d)},A_{0})\\
\delta(q_{(e,d)},\epsilon ,Z) &=& (q_{(e'',d)},Z).
\end{eqnarray*}

This subcase is largely the same as the previous, with the extra caveat that $A_0$ is the only stack symbol that may never be pushed to the stack. This will happened, from the perspective of contractions, if the only vertex of $T'$ mapping to the root of $T$ is the root of $T'$. We therefore have to be a bit careful to make sure the last two cases above are written separately.

For our penultimate subcase, we assume that $v = 0$ is the root, and that $e'$ is the smallest unvisited outgoing edge with head $v'$.

\begin{eqnarray*}
\delta(q_{(e,d)},u_0,A_0) &=& (q_{(e,d)},A_0A_0)\\
\delta(q_{(e,d)},u_0,Z) &=& (q_{(e,d)},A_0Z)\\
\delta(q_{(e,d)},d_0,A_0) &=& (q_{(e,d)},\epsilon)\\
\delta(q_{(e,d)},u_{v'},\epsilon) &=& (q_{(e',u)},A_{v'})
\end{eqnarray*}

Finally, assume that $v$ is the root, and that all outgoing edges of $v$ have been visited. Then there is nothing left to be done but resolve the symbols $A_0$ and move on to the final state.

\begin{eqnarray*}
\delta(q_{(e,d)},u_0,A_0) &=& (q_{(e,d)},A_0A_0)\\
\delta(q_{(e,d)},u_0,Z) &=& (q_{(e,d)},A_0Z)\\
\delta(q_{(e,d)},d_0,A_0) &=& (q_{(e,d)},\epsilon)\\
\delta(q_{(e,d)},\epsilon , Z) &=& (q_{f},\epsilon)
\end{eqnarray*}

We observe that, given any partial input, the transition function has at most one possible move. In particular, this PDA is unambiguous. It therefore remains to prove that the language of this PDA is $\mathcal{L}_T$. We proceed by induction on the number of edges of $T$.

In the case wherein $T$ is a single point, the language $\mathcal{L}_T$ is clearly seen to be the language of Dyck paths, whereas the PDA $P_T$ is easily seen to precisely agree with the PDA of Example \ref{cataex}. Assume then that $P_{T'}$ has associated language $\mathcal{L}_{T'}$ for all trees $T'$ with $<  n$ edges for some $n \geq 1$, and let $T$ be a planar rooted tree with $n$ edges. Write the planar rooted subtrees attached to the root of $T$ as $T_1,\ldots,T_r$, ordered in the natural way. Then by induction, as well as the definition of $P_T$, we see that $\mathcal{L}(P_T)$ is the language of words $w$ such that there exists some Dyck path $e = e_1e_2\cdots e_m$ on the alphabet $\{u_0,d_0\}$, as well as words $w_{T_i} \in \mathcal{L}_{T_i}$ with
\[
w = e_1\cdots e_{i_1} w_{T_1} e_{i_1+1} \cdots e_{i_2} w_{T_2} e_{i_2+1} \cdots e_m.
\]
Note that the respective alphabets we are using for the words $w_{T_i}$ are on the symbols $\{u_j,d_j\}$ where the permitted $j$ are determined by the vertices appearing in the respective subtrees $T_i$. It is obvious that this decomposition describes the words of $\mathcal{L}_T$, as desired.
\end{proof}

\begin{example}
To make things a bit more concrete, we fully describe the PDA $P_T$ in the case wherein $T$ is the tree that looks like the letter Y, with root on the bottom leaf. In this case vertices are numbered $0,1,2$ and $3$, in depth-first fashion, while we write our edges as $e_1,e_2,e_3$. Here, the index of the edge indicates the endpoint of the edge further from the root. Then we have
\begin{itemize}
\item $Q_T = \{q_0,q_{(e_1,u)}, q_{(e_2,u)},q_{(e_2,d)},q_{(e_3,u)},q_{(e_3,d)},q_{(e_1,d)},q_f\}$
\item $\Sigma_T = \{u_0,u_1,u_2,u_3,d_0,d_1,d_2,d_3\}, \Gamma_T = \{Z,A_0,A_1,A_2,A_3\}$
\end{itemize}
The complete list of our transition rules are given as follows:
\allowdisplaybreaks
\begin{eqnarray*}
\delta(q_0,u_0,\epsilon) &=& (q_0,A_0)\\
\delta(q_0,d_0,A_0) &=& (q_0,\epsilon)\\
\delta(q_0,u_1,\epsilon) &=& (q_{(e_1,u)},A_1)\\
\delta(q_{(e_1,u)},u_1,A_1) &=& (q_{(e_1,u)},A_1A_1)\\
\delta(q_{(e_1,u)},d_1,A_1) &=& (q_{(e_1,u)},\epsilon)\\
\delta(q_{(e_1,u)},u_2,A_1) &=& (q_{(e_2,u)},A_2A_1)\\
\delta(q_{(e_2,u)},u_2,A_2) &=& (q_{(e_2,u)},A_2A_2)\\
\delta(q_{(e_2,u)},d_2,A_2) &=& (q_{(e_2,u)},\epsilon) \\
\delta(q_{(e_2,u)},\epsilon , A_1) &=& (q_{(e_2,d)},A_1)\\
\delta(q_{(e_2,d)},u_1,A_1) &=& (q_{(e_2,d)},A_1A_1)\\
\delta(q_{(e_2,d)},d_1,A_1) &=& (q_{(e_2,d)},\epsilon)\\
\delta(q_{(e_2,d)},u_3,A_1) &=& (q_{(e_3,u)},A_3A_1)\\
\delta(q_{(e_3,u)},u_3,A_3) &=& (q_{(e_3,u)},A_3A_3)\\
\delta(q_{(e_3,u)},d_3,A_3) &=& (q_{(e_3,u)},\epsilon)\\
\delta(q_{(e_3,u)},\epsilon , A_1) &=& (q_{(e_3,d)},A_1)\\
\delta(q_{(e_3,d)},u_1,A_1) &=& (q_{(e_3,d)},A_1A_1)\\
\delta(q_{(e_3,d)},d_1,A_1) &=& (q_{(e_3,d)},\epsilon)\\
\delta(q_{(e_3,d)},\epsilon , A_0) &=& (q_{(e_1,d)},A_0)\\
\delta(q_{(e_3,d)},\epsilon , Z) &=& (q_{(e_1,d)},Z)\\
\delta(q_{(e_1,d)},u_0,A_0) &=& (q_{(e_1,d)},A_0A_0)\\
\delta(q_{(e_1,d)},u_0,Z) &=& (q_{(e_1,d)},A_0Z)\\
\delta(q_{(e_1,d)},\epsilon,Z) &=& (q_f,\epsilon).
\end{eqnarray*}

In words: in the initial state, the PDA can process three letters: $u_0,d_0$ and $u_1$. In the first case, the symbol $A_0$ is pushed onto the stack, while in the second case $A_0$ is popped from the stack. In the third case, we jump to the first non-initial state, and push the symbol $A_1$ onto the stack. To relate this to the context of planar contractions, we known that the root of $T'$ must map to the root of $T$. At this point we traverse $T'$ using the usual Dyck path method, at each step keeping track of what vertex of $T$ the contraction is mapping our current vertex of $T'$ to. In particular, at the beginning we map everything to the root of $T'$, until we step to the first vertex of $T'$ which maps to the first non-root vertex of $T$ (in the depth-first order). At this point, we have entered the regime of the word $w_\phi$ where the symbols $u_1$ and $d_1$ become active, while $u_0$ and $d_0$ become inactive. This will remain the case until we close off all of the symbols $u_1$ (as well as any intermediate symbols corresponding to the vertices of $T$ accessible from the vertex 1 without passing through the root). That is to say, until the top of the stack is either the symbol $A_0$ or $Z$. Here our PDA will have the option to move into the section of the word corresponding to a region of $T'$ which is once more being sent to the root. In this region, while we are free to use the symbols $u_0$ and $d_0$, we have to be careful not to suddenly begin reusing the symbol $u_1$. Indeed, once we have stepped back to the root in $T$, it would be a violation of the definition of contraction to return to the first vertex. This is why our states not only record which vertex of $T$ we are currently mapping to, but also whether we have just entered this regime from below, or above.
\end{example}

The construction of the PDA in Proposition \ref{almostUCF} inspires the following definition.

\begin{definition}
Let $w_{\phi}$ be a word in $\mathcal{L}_T$, and let $(e,u)$ (resp. $(e,d)$) be a pair of an edge and a direction corresponding to a state in the PDA of Proposition \ref{almostUCF}. Then the letters appearing in the word $w_{\phi}$ which are processed by this PDA whilst in the state corresponding to $(e,u)$ (resp. $(e,d)$) comprise what we call the \textbf{$(e,u)$ (resp. $(e,d)$) section of the word $w_{\phi}$}. The portion of $w_{\phi}$ which is parsed in the $q_0$-state of the PDA will be called the \textbf{initial section}. When the specific state of the PDA is not relevant to what is being discussed, we will often times just refer to the \textbf{sections of the word $w_{\phi}$}.
\end{definition}

\begin{example}
If we take $w_\phi = u_0u_0d_0u_1u_1d_1d_1d_0$, then the initial section of $w_\phi$ is the subword $u_0u_0d_0u_1$, while the $(e,u)$-section is $u_1d_1d_1$, and the $(e,d)$-section is $d_0$.
\end{example}

In accordance with Definition \ref{contextfreecat}, we have to verify that the order ideals of the poset $|\mathcal{PT}^{op}_T|$ are also unambiguous context-free languages, as well as the condition that our norm agrees with the length function on the language. The latter of these two goals is immediate from the relevant definitions.

\begin{lemma}\label{normisok}
The norm of $\nu$ defined in the statement of Theorem \ref{mainlangthm} respects the length on $\mathcal{L}_T$.
\end{lemma}

Before we can begin the proof of our final required statement,we introduce some notation that will be useful.

\begin{definition}
Let $w_\phi,w_{\phi'} \in \mathcal{L}_T$ be two words. We say that $w_\phi$ is \textbf{strongly contained in $w_{\phi'}$} if $w_\phi$ is a subword of $w_{\phi'}$, and whenever $u_i, d_i$ are a pair in $w_{\phi}$ - that is this $d_i$ is the alphabet symbol whose reading pops the original contribution of $u_i$ from the stack of $P_T$ - they are also a pair in $w_{\phi'}$. For instance, while $uudd$ appears as a subword of $udududud$, it is not strongly contained in this word.
\end{definition}

\begin{remark}
Counting patterns in Dyck paths is a relatively new field which seems to have many results analogous to the much more classical setting of counting patterns in permutations. See \cite{BBFGPW} for a treatment of these results. In this paper, we will be concerned with patterns that strongly appear in the word, as in the above definition. Our goal will be to show that the language of Dyck paths strongly containing any fixed pattern is actually unambiguous and context-free.
\end{remark}

\begin{proposition}
Let $T$ be a fixed planar tree, and let $I$ be an order ideal of $|\mathcal{PT}_T^{op}|$. Then the language associated to $I$ is unambiguous and context-free.
\end{proposition}

\begin{proof}
We first prove the proposition in the case where the order ideal is principal. In particular, we assume that
\[
I = (\phi) = \{\phi' \mid \phi' = \phi'' \circ \phi \text{ for some $\phi''$.}\}
\]
Translating everything through various definitions and equivalences, our goal in this proposition is to prove the following. We must show that the language
\[
\{w_{\phi'} \mid w_{\phi'} \text{ strongly contains $w_{\phi}$}\} \subseteq \mathcal{L}_T,
\]
is unambiguous and context-free.

Consider the PDA with component parts given by
\begin{itemize}
\item $Q = \{q^i_{0},q^{i_{(e,u)}}_{(e,u)},q^{i_{(e,d)}}_{(e,d)}, q_{found}, q_{f} \}$
\item $\Gamma = \{Z, (w_{\phi,j})\}$,
\end{itemize}
where $i$ ranges from 0 to the size of the initial section of $w_\phi$, and each $i_{(e,u)}$ or $i_{(e,d)}$ range from 0 to the length of the $(e,u)$ or $(e,d)$ section of the word $w_\phi$, respectively.

In words, the states of the PDA will encode both the section of the word we are currently parsing, as well as how much of $w_{\phi}$ we have observed thus-far. Our stack symbols include the initial symbol $Z$ as well as symbols corresponding to the subwords of $w_{\phi}$ comprised of the first $j$ letters for each $0 \leq j \leq l(w_{\phi})$.

We describe the transitions of this PDA in a particular example, and then discuss how they generalize. Let $T$ be a single edge, and $w_{\phi} = u_0d_0u_1u_1d_1d_1$. In this case our PDA has states given by 
\[
Q = \{q_{0}^0,q_{0}^1,q_{0}^2,q_{(e,u)}^0,q_{(e,u)}^1,q_{(e,u)}^2,q_{(e,u)}^3, q_{(e,d)}^0,q_{found},q_{f}\},
\]
while our stack alphabet is given by
\[
\Gamma = \{Z,(U_0),(U_0D_0),(U_0D_0U_1),(U_0D_0U_1U_1),(U_0D_0U_1U_1D_1),(U_0D_0U_1U_1D_1D_1))\}
\]
Our transition function is given as follows
\allowdisplaybreaks
\begin{eqnarray}
\delta(q_{0}^0,u_0,Z) &=& (q_{0}^1,(U_0)Z) \notag\\
\delta(q_{0}^1,u_0,\epsilon) &=& (q_{0}^1,(U_0))\notag \\
\delta(q_{0}^1,d_0,(U_0)) &=& (q_{0}^2,\epsilon)\notag\\
\delta(q_{0}^2,u_0,\epsilon) &=& (q_{0}^2,(U_0D_0))\notag\\
\delta(q_{0}^2,d_0,(U_0)) &=& (q_{0}^2,\epsilon)\notag\\
\delta(q_{0}^2,d_0,(U_0D_0)) &=& (q_{0}^2,\epsilon)\notag\\
\delta(q_{0}^2,u_1,\epsilon) &=& (q_{(e,u)}^0,(U_0D_0U_1))\notag\\
\delta(q_{(e,u)}^0,u_1,\epsilon) &=& (q_{(e,u)}^1,(U_0D_0U_1U_1))\notag\\
\delta(q_{(e,u)}^1,u_1,\epsilon) &=& (q_{(e,u)}^1,(U_0D_0U_1U_1))\notag\\
\delta(q_{(e,u)}^1,d_1,(U_0D_0U_1U_1)) &=& (q_{(e,u)}^2,\epsilon)\notag\\
\delta(q_{(e,u)}^2,u_1,\epsilon) &=& (q_{(e,u)}^2,(U_0D_0U_1U_1D_1)) \label{strangemove}\\ 
\delta(q_{(e,u)}^0,d_1,(U_0D_0U_1U_1D_1)) &=& (q_{(e,u)}^2,\epsilon) \label{strangemove2}\\
\delta(q_{(e,u)}^2,d_1,(U_0D_0U_1U_1)) &=& (q_{(e,u)}^3,\epsilon)\notag\\
\delta(q_{(e,u)}^2,d_1,(U_0D_0U_1)) &=& (q_{(e,d)}^0,\epsilon)\notag\\
\delta(q_{(e,u)}^3,u_1,\epsilon)) &=& (q_{(e,u)}^3,(U_0D_0U_1U_1D_1D_1))\notag\\
\delta(q_{(e,u)}^3,d_1,(U_0D_0U_1U_1D_1D_1)) &=& (q_{(e,u)}^3,\epsilon)\notag\\
\delta(q_{(e,u)}^3,d_1,(U_0D_0U_1U_1)) &=& (q_{(e,u)}^3,\epsilon)\notag\\
\delta(q_{(e,u)}^3,d_1,(U_0D_0U_1U_1D_1)) &=& (q_{(e,u)}^3,\epsilon)\notag\\
\delta(q_{(e,u)}^3,d_1,(U_0D_0U_1)) &=& (q_{(e,u)}^3,\epsilon)\notag\\
\delta(q_{(e,d)}^0,\epsilon,\epsilon) &=& \delta(q_{found},\epsilon,\epsilon)\notag\\
\delta(q_{found},u_0,\epsilon) &=& (q_{found},(U_0D_0U_1U_1D_1D_1))\notag\\
\delta(q_{found},d_0,(U_0D_0U_1U_1D_1D_1)) &=& (q_{found},\epsilon)\notag\\
\delta(q_{found},d_0,(U_0D_0)) &=& (q_{found},\epsilon)\notag\\
\delta(q_{found},d_0,(U_0)) &=& (q_{found},\epsilon)\notag\\
\delta(q_{found},\epsilon,Z) &=& (q_{f},\epsilon)\notag .
\end{eqnarray}

To summarize, the states of the PDA indicate both the section of $w_\phi$ which has been thus far detected, as well as the section of the input word currently being processed. We note that sections of $w_\phi$ must appear in the corresponding sections of the word being processed, and so there is nothing lost by partitioning our states this way. The stack symbols are meant to indicate how much of the word $w_\phi$ has been processed at the time the current letter is being read. This way, for instance, when certain symbols are popped from the stack, the PDA can determine when it needs to leave a certain state as the currently observed partial pattern can no longer be completed. On the other hand, this also allows for us to know whether or not a currently being read down symbol corresponds to the correct up symbol in so far as the pattern is concerned. In particular, it guarantees the copy of $w_\phi$ being detected is strongly included in the word, and not just a subword. The state $q_{found}$ is entered when the word $w_\phi$ has been fully detected, and we no longer have to worry about tracking exactly what is being read.

One notable pair of transitions is (\ref{strangemove}) and (\ref{strangemove2}). At this point in the PDA, the partial pattern $u_0d_0u_1u_1d_1$ has been detected, and the next letter read is a $u_1$. Whenever one follows a down move with an up move, and the desired pattern $w_\phi$ calls for a second down move, you have entered a section of your word which can no longer contribute to completing the partial pattern originally being observed. Therefore, the PDA must fall back to an earlier state and begin looking for a new pattern. If it finds this new pattern, then it enters the "found" state. Otherwise the symbol $(U_0D_0U_1U_1D_1)$ will eventually be popped from the stack, indicating to the PDA it must return to the original partial pattern and attempt to complete it.

Observe many states and transitions are not strictly necessary in this example. For instance, once two $u_1$ have been observed, it is impossible for us to not find our pattern. We present this example in this overly long way just to make it more clear how it generalizes. Finally, Observe that our PDA is unambiguous as for every input letter and top of the stack, there is at most one transition available.

Now that we have treated the case of a principal order ideal, we must treat the general case of an order ideal $I$. To begin, recall that Barter \cite{Barter} has already shown that the poset $|\mathcal{PT}_T|$ is Noetherian. In particular, all order ideals can be expressed as a finite union $I = \cup_{i = 1}^N (\phi_i)$. We therefore must prove
\[
\{w_{\phi} \mid w_{\phi} \text{ strongly contains at least one of the words $w_{\phi_i}$}\}
\]
is unambiguous and context-free. Because the list of desired patterns is finite, it is clear that one may modify the above so that our states record how much we have seen of each of the patterns independently. Our stack symbols will encode the currently observed partial patterns for each of the finitely many target patterns.
\end{proof}

\section{Applications of the main theorem}

In this section, we see certain concrete applications of the main Theorem \ref{mainthm}. Our first application involves counting a certain recursive invariant of Dyck words.

\begin{definition}
Given a Dyck path $w$, we define its \textbf{degree sequence} as the $(\frac{l(w)}{2}+1)-tuple$ $(\alpha^{(v)}_w)_{v \in T(w)}$ encoding the degree sequence of the associated planar rooted tree $T_p(w)$. If $\alpha_w$ is the degree sequence of some Dyck path, then we write
\[
||\alpha_w||_\ast := \sum_{v \in T(w)} \binom{\alpha_w^{(v)}}{2}
\]
for the \textbf{star-norm} of $\alpha_w$.
\end{definition}

Our first application involves the generating function of the star-norm.

\begin{theorem}\label{firstapp}
The generating function
\[
H_{star}(t) := \sum_{w} ||\alpha_w||_\ast t^{l(w)/2}
\]
is algebraic.
\end{theorem}

\begin{proof}
We will encode the series $H_{star}(t)$ as the Hilbert series of some finitely generated $\mathcal{PT}^{op}$-module. Let $T$ be a tree, and let $\cone(T)$ denote the graph obtained by adding a single vertex and connecting it to every vertex of $T$. For instance, the cone of a single edge is a triangle, while the cone of a path with three vertices is two triangles glued along an edge. In \cite{PR-trees}, a finitely generated $\mathcal{PT}^{op}$-module $M$ is constructed with the property that for any Dyck path $w$,
\[
M(T_p(w)) = H_1(\UConf_2(\cone(T_p(w));\Q)),
\]
where $\UConf_2(\cone(T_p(w)))$ is the two particle unordered configuration space of the graph $\cone(T_p(w))$ (see \cite{ADK}\cite{Ramos-Graph}\cite{Farley}, for instance). It is also shown in \cite{PR-trees} that
\[
\dim_\Q(H_1(\UConf_2(\cone(T_p(w));\Q))) = l(w)/2 + ||\alpha_w||_\ast.
\]
Therefore, the Hilbert series of the module $M$ is given by
\[
H_M(t) = \sum_w (l(w)/2 + ||\alpha_w||_\ast) t^{l(w)/2} = \sum_w (l(w)/2)t^{l(w)/2} + H_{star}(t).
\]
On the other hand, one can see that 
\[
\sum_w (l(w)/2)t^{l(w)/2} = t \cdot \frac{\partial}{\partial t}(\sum_w t^{l(w)/2}).
\]
It is classically known that
\[
\sum_w t^{l(w)/2} = \frac{1- \sqrt{1-4t}}{2t},
\]
whence
\[
\frac{\partial}{\partial t}(\sum_w t^{l(w)/2}) = \frac{(-2 t - \sqrt{1 - 4 t} + 1)}{(2 t^2\sqrt{1 - 4 t})}
\]
is an algebraic function. It follows that $\sum_w (l(w)/2)t^{l(w)/2}$ is algebraic, and the same must be true of $H_{star}(t)$.
\end{proof}

\begin{remark}
The observation that the derivative of $\sum_{w} (l(w)/2) t^{l(w)/2}$ is once again algebraic is really a specific case of a much more general phenomenon related to D-finite series. See \cite{BD} for more on this.
\end{remark}

Note that the above Theorem is actually just applying the techniques of this paper to the case of the first homology of tree configuration spaces. In fact, the results of \cite{PR-trees} tell us that all of the homology groups are finitely generated as $\mathcal{PT}^{op}$-modules. In particular, the same proof technique as the above can be used to prove a variety of more complicated numerical invariants of degree sequences of Dyck paths have algebraic generating functions. See \cite{Ramos-Graph} for what these formulas look like.

For our second application, we look to counting subtrees of a given tree. The generating function for counting subtrees is of considerable interest in computer science (see \cite{Rus}, For instance). In this work we consider the generating function of the following collections of numbers.

\begin{definition}
Let $w$ be a Dyck path with associated planar rooted tree $T_p(w)$, and let $l \geq 2$ be fixed. Then we set $s_l(w)$ to be the invariant
\[
s_l(w) = \begin{cases} \# (\text{planar rooted subtrees of $T_p(w)$}) & \text{ if $T_p(w)$ has no more than $l$ leaves}\\ 0 &\text{ otherwise.}\end{cases}
\]
\end{definition}

\begin{example}
For instance, if $l = 2$, then those $w$ for which $s_l(w) \neq 0$ precisely correspond to planar rooted paths. If $T_p(w)$ is a planar rooted path, then each subtree, that is not a single vertex, is in bijection with unordered pairs of vertices of $T_p(w)$. In particular,
\[
s_2(w) = \begin{cases} (l(w)/2+1) + \binom{l(w)/2+1}{2} &\text{ if $T_p(w)$ is a path}\\ 0 &\text{ otherwise}\end{cases} = \begin{cases} \binom{l(w)/2+2}{2} &\text{ if $T_p(w)$ is a path}\\ 0 &\text{ otherwise.}\end{cases}
\]
Thus,
\[
\sum_w s_2(w)t^{l(w)/2} = 1 + \sum_{n \geq 1} n \binom{n+2}{2}t^{n}
\]
where the factor of $n$ is the number of Dyck paths corresponding to paths with $n\geq 1$ edges, and the plus one comes from the case of the empty path corresponding to $T_p(w)$ being a single point. This implies that the generating function for $s_2(w)$ is rational. In general, we will see that $s_l(w)$ has an algebraic generating function.
\end{example}

\begin{theorem}
Let $l \geq 2$ be fixed. Then the generating function
\[
H_l(t) := \sum_w s_l(w)t^{l(w)/2}
\]
is algebraic.
\end{theorem}

\begin{proof}
We encode $H_l(t)$ as the Hilbert series of some $\mathcal{PT}^{op}$-module. In \cite{PR-trees}, Proudfoot and the author construct a finitely generated $\mathcal{PT}^{op}$-module $M$ such that
\[
\dim_\Q(M(T)) = \begin{cases} KL_1(\cone(T)) &\text{ if $T$ has $\leq l$ leaves}\\ 0 & \text{ otherwise,}\end{cases}
\]
where $KL_1(\cone(T))$ is the first Kahzdan-Lusztig coefficient of the graphical matroid associated to $\cone(T)$ (see \cite{EPW}). For our purposes, what is important is the fact that for any matroid $\mathcal{M}$,
\[
KL_1(\mathcal{M}) = \#(\text{codimension 1 flats of $\mathcal{M}$}) - \#(\text{dimension 1 flats of $\mathcal{M}$}).
\]
The graphical matroid of $\cone(T)$ has a 1 dimensional flat for every edge, and therefore
\[
\#(\text{dimension 1 flats of $\cone(T)$}) = |E_T| + |V_T| = 2|E_T| + 1.
\]
On the other hand, the number of codimension 1 flats of $\cone(T)$ are easily seen to be in bijection with subtrees of $T$. Thus,
\[
H_M(t) = \sum_w \dim_\Q(M(T_p(w)))t^{l(w)/2} = \sum_{T_p(w) \text{ has $\leq l$ leaves}} (s_l(w) -(l(w) + 1))t^{l(w)/2}.
\]

On the other hand, we can define a $\mathcal{PT}^{op}$-module $N$ such that,
\[
N(T) := \Q E_{\cone(T)},
\]
the vector space with basis indexed by the edges of $\cone(T)$. $N$ is finitely generated by a single edge and a single vertex. Moreover, it contains a submodule $N'$ which is generated by the pieces $N(T)$, where $T$ has strictly more than $l$ leaves. The quotient $N/N'$ has the property that
\[
N/N'(T) = \begin{cases} \Q E_{\cone(T)} &\text{ if $T$ has $\leq l$ leaves}\\ 0 &\text{ otherwise.}\end{cases}
\]
Therefore, the Hilbert series of $N/N'$ is precisely,
\[
\sum_{T_p(w) \text{ has $\leq l$ leaves}} (l(w) + 1)t^{l(w)/2}.
\]
Because $N$ is finitely generated, $N/N'$ is as well, and therefore this Hilbert series must be algebraic. This implies that our generating function is also algebraic, as desired.
\end{proof}

\end{document}